\documentclass[12pt]{article}
\usepackage[top=1in,bottom=1in,left=1in,right=1in]{geometry}

\usepackage{graphicx}
\usepackage{multicol,multirow}
\usepackage{amsmath,amssymb,amsfonts}
\usepackage{mathrsfs}
\usepackage{rotating}
\usepackage{appendix}
\usepackage[numbers]{natbib}
\usepackage{amsthm}
\usepackage{tikz-cd}
\usepackage{subcaption}
\usepackage{enumerate}

\newtheorem{theorem}{Theorem}[section]
\newtheorem{thmx}{Theorem}
\newtheorem{prop}[theorem]{Proposition}
\newtheorem{propx}[thmx]{Proposition} 
\newtheorem{lemma}[theorem]{Lemma}
\newtheorem{corollary}[theorem]{Corollary}
\newtheorem{corx}[thmx]{Corollary} 

\theoremstyle{remark}
\newtheorem{remark}[theorem]{Remark}

\numberwithin{equation}{section} 

\DeclareMathOperator{\Heis}{Heis}
\DeclareMathOperator{\Diff}{Diff}
\DeclareMathOperator{\Isom}{Isom}

\DeclareMathOperator{\Fix}{Fix}
\DeclareMathOperator{\Ind}{Ind}
\DeclareMathOperator{\proj}{proj}
\DeclareMathOperator{\graph}{graph}

\newcommand{\R}{{\mathbb R}}
\newcommand{\N}{{\mathbb N}}
\newcommand{\Z}{{\mathbb Z}}

\newcommand{\C}{{\mathbb C}}
\newcommand{\T}{{\mathbb T}}

\newcommand{\F}{\mathscr{F}}
\newcommand{\id}{{\text{id}}}
\newcommand{\G}{\Gamma}

\title{Smooth Models for Certain Fibered Partially Hyperbolic Systems}
\author{Meg Doucette\thanks{Department of Mathematics, University of Chicago \newline This material is based upon work supported by the National Science Foundation Graduate Research Fellowship Program under Grant No. DGE-1746045.}}
\date{July 8, 2023}

\begin{document}
\maketitle

\begin{abstract}
	We prove that under restrictions on the fiber, any fibered partially hyperbolic system over a nilmanifold is leaf conjugate to a smooth model that is isometric on the fibers and descends to a hyperbolic nilmanifold automorphism on the base. One ingredient is a result of independent interest generalizing a result of Hiraide: an Anosov homeomorphism of a nilmanifold is topologically conjugate to a hyperbolic nilmanifold automorphism. 
\end{abstract}

\section{Introduction}
	This paper establishes the existence of smooth models for certain fibered partially hyperbolic systems.   An ingredient in our proof is a classification of Anosov homeomorphisms on nilmanifolds.  We explain now through an example what some of these terms mean.
	
	The three-dimensional Heisenberg group is given by 
		$$
			\Heis:=
			\left\{ 
				A_{(x,y,z)}=
				\begin{pmatrix}
					1&x&z\\ 0&1&y\\0&0&1
				\end{pmatrix} 
				: \ x,y,z\in \R
			\right\}
		$$
	with matrix multiplication as the group operation. The group $\Heis$ is a nilpotent Lie group whose center consists of the matrices $A_{(0,0,z)}$, with $z\in \R$.  Any automorphism of $\Heis$ must preserve this center.
	
	Nilmanifolds are quotients of nilpotent Lie groups by discrete subgroups.  For the Heisenberg group, a compact quotient can be obtained as follows.  Let 
		$$
			\Gamma=
			\left\{
				(x,y,z)\in H : \ x,y,2z\in \Z
			\right\},
		$$
	where we use $(x,y,z)$ to denote the matrix $A_{(x,y,z)}$.   The quotient $\Heis /\Gamma$ is a compact nilmanifold, an example of a {\em Heisenberg nilmanifold}.  It is a fiber bundle over the $2$-torus ${\mathbb T}^2$ with fiber the circle ${\mathbb T}$, where the fibers lie in the ``$z$-direction," tangent to the center of the Lie algebra of $\Heis$.	 

	Any automorphism of $\Heis$ that preserves $\Gamma$ descends to a diffeomorphism of $\Heis/\Gamma$: since the automorphism preserves the center of $\Heis$, the quotient diffeomorphism preserves the bundle structure.	The interesting quotient diffeomorphisms are examples of {\em fibered partially hyperbolic diffeomorphisms}.  An example is the map  $f_0:\Heis/\Gamma \to \Heis/\Gamma$ given by
		$$
			f_0(x,y,z)=\left(2x+y,x+y,z+x^2+\frac{y^2}{2}+xy\right).
		$$
	Since $f_0$ preserves the smooth fibration ${\mathbb T}\hookrightarrow \Heis/\Gamma \twoheadrightarrow {\mathbb T}^2$, it induces a diffeomorphism of the base ${\mathbb T}^2$, in this case the hyperbolic linear automorphism $(x,y)\mapsto (2x+y,x+y)$.

	This example is partially hyperbolic, meaning that the tangent bundle $TN$ to $N=\Heis/\Gamma$ splits as a $df_0$-invariant direct sum $TN=E^s\oplus E^c\oplus E^u$ such that for all $x\in N$ and all unit vectors $v^s\in E^s(x), \ v^c\in E^c(x), $ and $v^u\in E^u(x)$, we have that
		$$
				\|d_xf(v^s)\|<\|d_xf(v^c)\|<\|d_xf(v^u)\| 
			\quad \text{ and } \quad
				\|d_xf(v^s)\|<1<\|d_xf(v^u)\|.
		$$
	That is, the center direction is dominated by the stable and unstable directions.

	Hammerlindl and Potrie proved that if $f\colon N\to N$ is partially hyperbolic and homotopic to $f_0$, then $f$ is {\em leaf conjugate} to $f_0$, meaning that there exists a homeomorphism $h:N\to N$, an $f$-invariant foliation $W^c_{f}$ tangent to the center direction of $f$, and an $f_0$-invariant foliation $W^c_{f_0}$ tangent to the center direction of $f_0$ such that $h$ maps center leaves of $f_0$ to center leaves of $f$ (i.e. $h(W^c_{f_0}(x))=W^c_f(h(x))$) and $h(f_0(W^c_{f_0}(x)))=f(h(W^c_{f_0}(x)))$ \cite{HammerlindlPotrie}. In this example, $f_0$ is a smooth model, and any partially hyperbolic diffeomorphism homotopic to it is leaf conjugate to it.

	In this paper, we consider the class of {\em fibered partially hyperbolic diffeomorphisms}.  These are partially hyperbolic diffeomorphisms, that, like $f_0$ above, have an integrable center bundle $E^c$, tangent to an invariant fibration by compact submanifolds.  More precisely $f\colon M\to M$, where $M$ is a closed Riemannian manifold, is a \textit{fibered partially hyperbolic diffeomorphism}\footnote{This is also known as a {\em{fibered partially hyperbolic system with $C^k$ fibers}}.} with ($C^k$, $k\geq 1$) fiber $X$ if there exists an $f$-invariant continuous fiber bundle $\pi:M\to B$ (for some manifold $B$) with $C^k$ fibers modeled on $X$, which are tangent to $E^c$ and such that the $k$-jets along fibers are continuous in $M$.\footnote{ In other words, $M$ is a continuous $X$-bundle with structure group $\Diff^k(X)$.} We say that a fibered partially hyperbolic system $(f,\pi,M)$ is $C^k$ if the fiber bundle $\pi:M\to B$ is $C^k$. Note the distinction between a $C^k$ fibered partially hyperbolic system and a fibered partially hyperbolic system with $C^k$ fibers: in a $C^k$ fibered partially hyperbolic system, we require that the bundle $\pi:M\to B$ is $C^k$, whereas in a fibered partially hyperbolic system with $C^k$ fibers, the bundle $\pi:M\to B$ is merely required to be continuous.

	Fibered partially hyperbolic systems form a rich class of dynamical systems. Beyond the relatively simple examples of skew products on trivial bundles, fibered systems appear as automorphisms of nilmanifolds and play a role in the construction of exotic partially hyperbolic systems (e.g. \cite{gogolev_new_2015}). The classification of fibered partially-hyperbolic diffeomorphisms up to leaf conjugacy is an open question in general, except in low dimension. Such diffeomorphisms appear in several rigidity contexts (e.g. \cite{avila_absolute_2022}, \cite{damjanovic_pathology_2021}, \cite{nitica_local_2001}).
	
	Here is our main result.

	\begin{thmx}\label{thm:smoothmodel}
		Let $f:M\to M$ be a fibered partially hyperbolic system with quotient a nilmanifold $B$ and $C^1$ fibers $F$ (where $F$ is a closed manifold). Suppose that the structure group of the $F$-bundle $M$ is $G\subset \Diff^1(F)$ and that there exists a Riemannian metric on $F$ and a subgroup $I$ of $\Isom(F)\cap G$ such that the inclusion $I\hookrightarrow G$ is a homotopy equivalence.
		
		Then $f$ is leaf conjugate to a $C^\infty$ fibered partially hyperbolic system $g:\widehat{M}\to \widehat{M}$ such that
		
		\begin{enumerate}
		\item The projection of the leaf conjugacy to $B$ is a map homotopic to the identity; \label{mainthm: proj leaf conj}
		\item the $F$-bundles $M$ and $\widehat M$ are isomorphic\footnote{This is implicit from the definition of leaf conjugacy, but we state it explicitly for clarity.}; \label{mainthm: isomorphic bundles}
		\item the structure group of $\widehat M$ is $\Isom(F)$; and \label{mainthm structure grp bullet}
		\item the projection of $g$ to $B$ is a hyperbolic nilmanifold automorphism. \label{mainthm: proj nilmanifold automorphism}
		\end{enumerate}
	\end{thmx}
	
	\begin{remark}
		\begin{itemize}
		\item 
			The base manifold $B$ in Theorem \ref{thm:smoothmodel} a priori might not have a smooth structure. In our results, we assume $B$ is a \textit{topological nilmanifold}, meaning that $B$ is homeomorphic to a nilmanifold. This case can easily be reduced to the case where $B$ is a nilmanifold by replacing the projection map for the bundle with the projection map composed with the homeomorphism. Thus, our results easily extend to a topological nilmanifold $B$.
		\item 
			\ref{mainthm structure grp bullet}. implies that there exists a smooth Riemannian metric on $\widehat M$ adapted to $g$ such that $g$ is isometric on fibers. This will be clear from the construction of $g$ in the proof of Theorem \ref{thm:smoothmodel}.
		\item 
			The fibered partially hyperbolic system $g:\widehat M \to \widehat M$ may act differently on the fibers than the original fibered partially hyperbolic system $f:M\to M$; that is, if $h:M\to \widehat M$ is the leaf conjugacy from Theorem \ref{thm:smoothmodel}, then $h\circ f$ and $g\circ h$ may not be homotopic. 
		\end{itemize}
	\end{remark}
	
	The main assumption in Theorem \ref{thm:smoothmodel} is that the structure group of the $F$-bundle contain a homotopy equivalent subgroup of $\Isom(F)$. We discuss the necessity of this assumption to our proof in Remark \ref{remark:assumption on structure group}. Finding circumstances where this assumption applies (and thus we can apply Theorem \ref{thm:smoothmodel}) comes down to studying the relationship between $\Diff^1(F)$, $\Isom(F)$, and their subgroups.\footnote{In the following discussion we often replace $\Diff^1(F)$ with $\Diff^\infty(F)$, which we can do since $\Diff^\infty(F)\hookrightarrow\Diff^1(F)$ is a homotopy equivalence (Proposition \ref{prop:InclusionDiffeoGrps}).} Namely, for which manifolds $F$ and which subgroups $G$ of $\Diff^\infty(F)$ is there a subgroup $H\subset \Isom(F)\cap G$ such that the inclusion $H\hookrightarrow G$ is a homotopy equivalence?
	
	Note that even without the assumption that the structure group of the $F$-bundle contain a homotopy equivalent subgroup of $\Isom(F)$, our argument gives that the initial fibered partially hyperbolic system $f:M\to M$ is leaf conjugate to an extension over a hyperbolic nilmanifold automorphism. For further discussion and details, see Remark \ref{remark:assumption on structure group}. Notably, in the case where the $F$-bundle from the fibered partially hyperbolic system $f:M\to M$ is trivial (i.e. $M=B\times F$), we can construct this extension over a hyperbolic nilmanifold automorphism to be partially hyperbolic.
	
	\begin{propx}\label{prop: trivial case main thm}
		Let $f:M\to M$ be a fibered partially hyperbolic system with quotient a nilmanifold $B$ and $C^1$ fibers $F$ (where $F$ is a closed manifold). Suppose that the $F$-bundle $M$ is trivial (i.e. that the $F$-bundle $M$ is isomorphic to $B\times F$).
		
		Then $f$ is leaf conjugate to a $C^\infty$ fibered partially hyperbolic system $g:\widehat{M}\to \widehat{M}$ such that \ref{mainthm: proj leaf conj}., \ref{mainthm: isomorphic bundles}, and \ref{mainthm: proj nilmanifold automorphism}. from Theorem \ref{thm:smoothmodel} hold.
	\end{propx}
	
	The proof of Proposition \ref{prop: trivial case main thm} is given at the end of Section \ref{sect:proof of main thm}.
	
	The following two corollaries come from answering the above question about subgroups $H\subset \Isom(F)\cap G$ for specific $F$ and $G$. They are by no means the only such corollaries. (For example, the conclusion of Corollary \ref{cor:StructureGroupConnectedComponent} holds for any $F$ such that $\Diff^1_0(F)$ is contractible.) 
	
	\begin{corx}\label{cor:StructureGroupDiff}
		Let $f:M\to M$ be a fibered partially hyperbolic system with quotient a nilmanifold $B$ and $C^1$ fibers $F$, where $F$ is 
			\begin{enumerate}
			\item a $n$-sphere $S^n$ for $n=1,2,3$, or \label{cor:fibersS^n}
			\item a hyperbolic $3$-manifold \label{cor:fibershyperbolic3}
			\end{enumerate}
		 Then $f$ is leaf conjugate to a $C^\infty$ fibered partially hyperbolic system $g:\widehat{M}\to \widehat{M}$, which induces a hyperbolic nilmanifold automorphism on the base.
	\end{corx}
	
	\begin{corx}\label{cor:StructureGroupConnectedComponent}
		Let $f:M\to M$ be a fibered partially hyperbolic system with quotient a nilmanifold $B$ and $C^1$ fibers $F$ and with structure group $\Diff^1_{0}(F)$. Suppose that $F$ is
			\begin{enumerate}
			\item the two or three torus, $\mathbb{T}^2, \ \mathbb{T}^3$ \label{cor:fibers torus},
			\item a hyperbolic surface, or \label{cor:fibershyperbolic2}
			\item a Haken 3-manifold \label{cor:haken}
			\end{enumerate}
		Then $f$ is leaf conjugate to a $C^\infty$ fibered partially hyperbolic system $g:\widehat{M}\to \widehat{M}$, which induces a hyperbolic nilmanifold automorphism on the base.
	\end{corx}
	
	\begin{remark}
		As noted before, these are by no means the only cases where Theorem \ref{thm:smoothmodel} can be applied to get results analogous to the results of Corollaries \ref{cor:StructureGroupDiff} and \ref{cor:StructureGroupConnectedComponent}. Notably, the conclusion of Corollary \ref{cor:StructureGroupDiff} holds for any fiber $F$ such that the inclusion $\Isom(F)\hookrightarrow \Diff(F)$ is a homotopy equivalence. For the sake of conciseness, we haven't listed in Corollary \ref{cor:StructureGroupDiff} more of the known examples of closed manifolds $F$ for which $\Isom(F)\hookrightarrow \Diff(F)$ is a homotopy equivalence. Other examples include lens spaces, prism and quaternionic manifolds, tetrahedral manifolds, octahedral manifolds, and icosahedral manifolds \cite{BamlerKleiner}.
	\end{remark}
	
	Theorem \ref{thm:smoothmodel} builds on previous work by Hirsch-Pugh-Shub and by Hammerlindl and Potrie. Hirsch, Pugh, and Shub \cite{HPS} proved that perturbations of fibered partially hyperbolic systems are fibered partially hyperbolic systems, and that the perturbed system is leaf conjugate to the original system. Hammerlindl \cite{HammerlindlTorus} proved that a partially hyperbolic diffeomorphism of $\T^3$ is leaf conjugate to a linear automorphism of $\T^3$. Hammerlindl and Potrie \cite{HammerlindlPotrie} proved an analogous result for partially hyperbolic diffeomorphisms of 3-dimensional nilmanifolds.
	
	Corollaries \ref{cor:StructureGroupDiff} and \ref{cor:StructureGroupConnectedComponent} give specific examples of cases where we can apply Theorem \ref{thm:smoothmodel} in dimensions 1, 2, and 3, and Proposition \ref{prop: trivial case main thm} gives an analogue of Theorem \ref{thm:smoothmodel} in the case of a trivial bundle. Note that Corollary \ref{cor:StructureGroupDiff} shows that the hypothesis of Theorem \ref{thm:smoothmodel} holds for all fibered partially hyperbolic systems with one dimensional fiber. Corollaries \ref{cor:StructureGroupDiff} and \ref{cor:StructureGroupConnectedComponent} also show that the conclusion of Theorem \ref{thm:smoothmodel} holds for surface bundles, although in the cases of surfaces other than $S^2$, an added assumption on the structure group of the bundle is needed. The question remains of whether, and in what conditions for other types of fibers in nontrivial bundles. 
	
	\begin{remark}
		If for the fibered partially hyperbolic system $f:M\to M$, $\dim E^s = 1$ (or $\dim E^u =1$), then the base space $B$ will always be a nilmanifold \cite{BohnetCodim1}. This allows us to replace the assumption that the quotient of the fibered partially hyperbolic system $f:M\to M$ in Theorem \ref{thm:smoothmodel} or in Proposition \ref{prop: trivial case main thm} with the assumption that $\dim E^s=1$ (or $\dim E^u =1$).
	\end{remark}

	One ingredient in the proof of Theorem \ref{thm:smoothmodel}, which is of independent interest, is the following generalization of the works of Franks-Manning and of Hiraide to Anosov homeomorphisms of nilmanifolds. 

		\begin{thmx}[{\cite[Theorem 2(1)]{Sumi}}]\label{thm:HiraideNilmanifolds}
			An Anosov homeomorphism of a nilmanifold is topologically conjugate to a hyperbolic nilmanifold automorphism via a conjugacy that is homotopic to the identity.
		\end{thmx}
		
	This theorem was originally proved by Sumi in \cite{Sumi}. A proof of Theorem \ref{thm:HiraideNilmanifolds}, which follows the same structure as Sumi and Hiraide's proofs in \cite{Sumi} and \cite{Hiraide_tori}, is provided in Section \ref{sec:HiraideGeneralization} for the sake of completeness. 
	
	{\em{Anosov homeomorphisms}} are generalizations of Anosov diffeomorphisms. Many of the important properties of Anosov diffeomorphisms come directly from the fact that Anosov diffeomorphisms are expansive and have the shadowing property. A homeomorphism $f:X\to X$ of a metric space is \textit{expansive} if there exists a constant $c>0$ such that for all $x,y\in X$, if $d(f^n(x),f^n(y))<c$ for all $n\in \Z$ then $x=y$. Such a constant $c$ is called an \textit{expansive constant} for $f$. 
	
	The shadowing property says that we can approximate \textit{pseudo-orbits} by actual orbits. More formally, a sequence of points $\{x_i\}_{i\in \Z}\subset X$ is called a \textit{$\delta$-pseudo-orbit} if $d(f(x_i),x_{i+1})<\delta$ for all $i\in \Z$. A point $z\in X$ is said to \textit{$\varepsilon$-shadow} a sequence of points, $\{x_i\}_{i\in \Z}$, if $d(f^i(z),x_i)<\varepsilon$ for all $i\in \Z$. We say that $f$ has the \textit{shadowing property} if for any $\varepsilon>0$, there exists $\delta>0$ such that any $\delta$-pseudo-orbit is $\varepsilon$-shadowed by a point in $X$. An expansive homeomorphism with the shadowing property is known as an \textit{Anosov homeomorphism}. 
	
	In contrast to some other weakenings of Anosov diffeomorphisms (e.g. hyperbolic homeomorphisms \cite[Section IV.9]{mane_ergodic_1987}), Anosov homeomorphisms are not assumed to have invariant foliations.
	
	In Section \ref{sec:HiraideGeneralization}, we examine some of the similarities between Anosov homeomorphisms and Anosov diffeomorphisms and give a proof Theorem \ref{thm:HiraideNilmanifolds}.
	
	Theorem \ref{thm:HiraideNilmanifolds} is useful in the proof of our main theorem because given a fibered partially hyperbolic diffeomorphism $f:M\to M$ with associated bundle $\pi:M\to B$, the map induced by $f$ on $B$ is an Anosov homeomorphism. (This follows from a result of Bohnet and Bonatti \cite{BohnetBonatti}, as we will explain in Section \ref{sect:proof of main thm}.)
	
	After proving Theorem \ref{thm:HiraideNilmanifolds} in Section \ref{sec:HiraideGeneralization}, we spend the rest of the paper proving Theorem \ref{thm:smoothmodel}. The proof is split into four parts. Due to the fact that much of the proof will take place in the quotient leaf space $B$, we denote the partially hyperbolic diffeomorphism on $M$ by $\hat f:M\to M$. We denote the homeomorphism that $\hat f$ descends to on $B$ by $f:B\to B$.
	
	The strategy of the proof is to first construct,
		\begin{itemize}
		\item 
			a conjugacy $h:B\to B$ between $f:B\to B$ and a hyperbolic nilmanifold automorphism $A:B\to B$, and 
		\item
			a smooth $F$-bundle $\widehat{M}$ over $B$ that is isomorphic to the original $F$-bundle $M$.
		\end{itemize}
	Then lift 
		\begin{itemize}
		\item 
			the conjugacy $h:B\to B$ to a homeomorphism $\hat h: \widehat M\to \widehat M$, and 
		\item 
			the hyperbolic nilmanifold automorphism $A:B\to B$ to a partially hyperbolic diffeomorphism $g:\widehat M \to \widehat M$.
		\end{itemize}
		
	The construction of the conjugacy $h:B\to B$ and the hyperbolic nilmanifold automorphism $A:B\to B$ relies almost entirely on Theorem \ref{thm:HiraideNilmanifolds}. The construction of the smooth bundle $\widehat M$ relies on tools developed in Section \ref{sect: fiber bundles}. Lifting the conjugacy $h:B\to B$ to a homomorphism $\hat h:\widehat M\to \widehat M$ takes place in Section \ref{sect: lifting conjugacy}, and finally lifting $A$ to a partially hyperbolic diffeomorphism $g$ relies on tools developed in Section \ref{sect: lifting Anosov automorphism}. The entire proof of Theorem \ref{thm:smoothmodel} is given in Section \ref{sect:proof of main thm}.
	
\section*{Acknowledgments}
	
	The author thanks Amie Wilkinson for helpful discussions and critical comments on drafts of the paper, as well as for her continued support. The author thanks Danny Calegari for helpful discussions on the topology and geometry of fiber bundles and on diffeomorphism groups. The author thanks Kathryn Mann for helpful discussions on diffeomorphism groups and on three-manifolds. The author thanks Rafael Potrie, Pablo Carrasco, Jon DeWitt, and Daniel Mitsutani for critical comments on earlier drafts of this paper. The author is also thankful to the anonymous referees for carefully reading the manuscript and for providing numerous helpful comments. 

\section{Preliminaries about Fiber Bundles}\label{sect: fiber bundles}

	The goal of this section is to provide several results about fiber bundles that are necessary for the proof of Theorem \ref{thm:smoothmodel}.

	The structure of a fiber bundle is given by its transition functions. Given a continuous fiber bundle $\pi:E\to B$ with fiber $F$ and structure group $G$, let $\left\{ (U_i,\phi_i:\pi^{-1}(U_i)\to U_i\times F) \right\}$ be locally trivializing charts for $E$. The functions $\tau_{ij}:=\phi_i\circ \phi_j^{-1}:U_i\cap U_j\to G$ are {\em{transition functions}} for the bundle $E$.
	
	All the data in a fiber bundle is contained in the transition functions. More precisely,
	
	\begin{lemma}[Fiber bundle construction theorem]\label{lem:fbconstruction}
		Let $F, B$ be $C^k$ manifolds ($0\leq k\leq \infty$), and let $G$ be a topological group with the structure of a $C^k$ manifold that has a $C^k$ left action on $F$. Given an open cover $\{U_i\}_{i\in A}$ of $B$ and a set of $C^k$ functions $\tau_{ij}:U_i\cap U_j\to G$ such that the cocycle condition, $\tau_{ij}(x)\tau_{jk}(x)=\tau_{ik}(x)$ holds for all $x\in U_i\cap U_j\cap U_k$. Then there exists a $C^k$ $F$-bundle $\pi:E\to B$ with transition functions $\tau_{ij}$.
	\end{lemma}
	
	This means that any set of functions $\{\tau_{ij}:U_i\cap U_j\to G\}$ that satisfy the cocycle condition, $\tau_{ij}(x)\tau_{jk}(x)=\tau_{ik}(x)$ holds for all $x\in U_i\cap U_j\cap U_k$, are transition functions for a fiber bundle. Additionally, the smoothness of a bundle is defined by the smoothness of the transition functions. This is a standard result about fiber bundles (e.g. see \cite[Section 5.3]{Husemoller}).
	
	Transition functions also classify fiber bundles up to isomorphism. Two fiber bundles are isomorphic if they have {\em{cohomologous}} transition functions. Two sets of transition functions (i.e. two sets of functions that satisfy the cocycle condition) $\{\tau_{ij}:U_i\cap U_j\to G\}$ and $\{ \tau'_{ij}:U_i\cap U_j\to G \}$ are said to be {\em{cohomologous}} if there exist continuous functions $t_i:U_i\to G$ such that $\tau'_{ij}(x)=t_i(x)^{-1}\tau_{ij}(x)t_j(x)$ for all $x\in U_i\cap U_j$.
	
	\begin{corollary}\label{cor:transitionfcns}
		Let $F,B$ be topological spaces and $G$ a topological group that has a continuous left action on $F$. Suppose $\pi:E\to B$ and $\tilde \pi:\tilde E\to B$ are continuous $F$-bundles with structure group $G$. Let $\{(U_i,{\phi_i}:\pi^{-1}(U_i)\to U_i\times F)\}$ and $\{(U_i,\widetilde{\phi_i}:\widetilde\pi^{-1}(U_i)\to U_i\times F)\}$ be locally trivializing charts for $E$ and $\tilde E$ respectively, and let $\{\tau_{ij}:U_i\cap U_j\to G\}$ and $\{\widetilde\tau_{ij}:U_i\cap U_j\to G\}$ be the associated transition functions for $\{(U_i,\phi_i)\}$ and $\{(U_i,\tilde\phi_i)\}$ respectively. Suppose there exist continuous functions $t_i:U_i\to G$ such that $\tilde \tau_{ij}(x)=t_i(x)^{-1}\tau_{ij}(x)t_j(x)$ for all $x\in U_i\cap U_j$. Then, $E$ and $\tilde E$ are isomorphic as $F$-bundles with structure group $G$. If, in addition, the left $G$ action on $F$ is faithful, then the converse holds.
	\end{corollary}
	
	For a proof of Corollary \ref{cor:transitionfcns}, see \cite[Chapter 5.2]{Husemoller}. 
	
	In the proof of Theorem \ref{thm:smoothmodel}, we change the structure group of the fiber bundles we are working with. We now make the notion of changing structure group precise. 
	
	Given a continuous homomorphism $\alpha:H\to G$ between two topological groups and a principal $H$-bundle $q:Q\to B$, we can construct a principal $G$-bundle from $Q$ and $\alpha$ in the following way. Consider the space
		$$
				Q\times_{\alpha,H}G
			:= 
				Q\times G/(x,g)\sim \{(x\cdot h,\alpha(h)g), h\in H\}
		$$
	Note that $Q\times_{\alpha,H}G$ has a free right $G$-action given by $[x,g]\cdot g'=[x,gg']$. This makes $Q\times_{\alpha,H}G$ a principal $G$-bundle. Note that the projection map for this bundle $Q\times_{\alpha,H}G\mapsto B$ is given by $[x,g]\mapsto q(x)$. 
	
	We say that the principal $H$-bundle $q:Q\to B$ induces the principal $G$-bundle $p:P\to B$ if $P\cong Q\times_{\alpha,H} G$. Additionally, we say that a principal $G$-bundle $p:P\to B$ admits \textit{a reduction of structure group from $G$ to $H$} if there exists a principal $H$-bundle $q:Q\to B$ that induces $P$.
	
	Now, we note the relationship between the transition functions between a bundle and a bundle it induces. Suppose that $t_{ij}:U_i\cap U_j\to H$ are transition functions for the principal $H$-bundle $q:Q\to B$. The transition functions for the induced bundle $Q\times_{\alpha,H}G$ are given by $\alpha\circ t_{ij}:U_i\cap U_j\to G$. 
	
	Next, we give conditions under which a principal bundle admits a reduction of structure group.
	
	\begin{lemma}\label{lem:reduction-structure-group}
		Suppose that $\alpha:H\to G$ is a homomorphism that is also a homotopy equivalence. Then any principal $G$-bundle admits a reduction of structure group from $G$ to $H$.
	\end{lemma}

	The proof of this lemma relies on the theory of classifying spaces. Let $G$ be a topological group. A principal $G$-bundle $\pi:EG\to BG$ is a {\em{universal}} principal $G$-bundle if for all CW-complexes $X$, the map from the set of homotopy classes of maps $X\to BG$ to the set of isomorphism classes of principal $G$-bundles over $X$, given by the map $f\mapsto f^*EG$ is a bijection. The base space of a universal principal $G$-bundle is known as a {\em{classifying space}} for $G$. 
	
	\begin{proof}
		See \cite[Chapter 6]{Husemoller}.
	\end{proof}
	
	\begin{remark}
		Note that if $H\subset G$ and the homomorphism $\alpha:H\to G$ is inclusion, then the reduction of structure group from $G$ to $H$ from Lemma \ref{lem:reduction-structure-group} has transition functions that are cohomologous to the transition functions of the original principal $G$ bundle.
	\end{remark}


	\begin{prop}\label{prop:InclusionDiffeoGrps}
		If $M$ is a closed, smooth manifold then the inclusion $\Diff^\infty(M)\hookrightarrow \Diff^1(M)$ is a homotopy equivalence.
	\end{prop}
	\begin{proof}
		For $0\leq k\leq \infty$, $\Diff^k(M)$ is an infinite-dimensional separable Fr\'echet space \cite[Section I.4.3]{Onishchik}. Since all infinite-dimensional separable Fr\'echet spaces are homeomorphic to the Hilbert space $\ell^2$ \cite{AndersonBing}, we see that $\Diff^k(M)$ is homeomorphic to the Hilbert space $\ell^2$. Since $\ell^2$ has the homotopy type of a CW-complex \cite{SmrekarYamashita}, we get that $\Diff^k(M)$ is homotopy equivalent to a CW-complex \cite{HatcherDiffeoGrps}. Thus by Whitehead's Theorem, to show that the inclusion $\iota:\Diff^\infty (M)\hookrightarrow \Diff^1(M)$ is a homotopy equivalence, it is sufficient to show that the induced map on homotopy groups $\iota_*:\pi_*(\Diff^\infty(M))\to \pi_*(\Diff^1(M))$ is an isomorphism. 
		
		To do this, we first recall that for any map in $\varphi\in C^1(M,M)$, we can find a smooth map that is arbitrarily close to and homotopic to $\varphi$ and that the choice of this map depends continuously on our original map (\cite[Theorems 6.21 and 6.28]{Lee}). This gives us a continuous map $\Phi:\Diff^1(M)\to C^\infty(M,M)$ such that for any $\varphi\in \Diff^1(M)$, $\Phi(\varphi)\simeq \varphi$ and $d(\varphi,\Phi(\varphi))<\varepsilon$. Since $\Diff^\infty(M)\subset C^\infty(M,M)$ is open (by the inverse function theorem), by choosing $\varepsilon$ small enough, we get that $\Phi(\Diff^1(M))\subset \Diff^\infty(M)$, so we can write $\Phi:\Diff^1(M)\to \Diff^\infty(M)$.
		
		Now, we show that the map induced by $\iota:\Diff^\infty(M)\to \Diff^1(M)$ on homotopy groups is an isomorphism. The map $\iota_*$ is surjective because any map $\phi:S^n\to \Diff^1(M)$ is homotopic to the map  $\Phi\circ \phi:S^n\to \Diff^\infty(M)$. To see that $\iota_*$ is injective, we take a map $\phi:S^n\to \Diff^1(M)$ that is null-homotopic. Let $h_t:S^n\to \Diff^1(M)$ be a null-homotopy for $\phi$. Then the map $\Phi\circ\phi:S^n\to \Diff^\infty(M)$ is homotopic to $\phi$ in $\Diff^1(M)$, and is null-homotopic in $\pi_n(M)$ via the homotopy $\Phi\circ h_t:S^n\to \Diff^\infty(M)$.
	\end{proof}

\section{Lifting the conjugacy on the leaf space}\label{sect: lifting conjugacy}

	\begin{lemma}\label{lem:liftingconjugacy}
		Let $F,M,\tilde M,B$ be closed Riemannian manifolds, and let $\pi:M\to B$ and $\widetilde{\pi}:\tilde M\to B$ be continuous isomorphic $F$-bundles. Let $h:B\to B$ be a homeomorphism that is homotopic to the identity. Then, $h:B\to B$ lifts to a homeomorphism $\tilde h:M\to \tilde M$.
	\end{lemma}
	
	\begin{proof}
		First, we note that $M$ and $h^*M$ are isomorphic bundles. This is because we by assumption $M$ and $\tilde M$ are isomorphic, and since $h\sim \id$, the bundle $h^* \tilde M$ is isomorphic to the bundle $\id^*\tilde M=\tilde M$. Let $\phi:\tilde M\to h^*\tilde M$ be a bundle isomorphism, i.e. the diagram in Figure \ref{fig:lifth1} commutes. From the definition of $h^*\tilde M$, we have that the commutative diagram in Figure \ref{fig:lifth2} commutes.\footnote{Note that $\proj_2:h^* \tilde M \to \tilde M$ is the projection onto the second coordinate from the definition of the pullback bundle $h^* \tilde M = \{ (b,x)\in B\times M:  h(b)=\tilde\pi (x) \}$.} Combining these diagrams gives us the commutative diagram from Figure \ref{fig:lifth3}. So the continuous map
			\begin{equation}\label{eqn:lifth}
				\tilde h:=\proj_2\circ \phi:M\to \tilde M
			\end{equation}
		is a lift of $h$. Since $\tilde h$ is a continuous injection, by invariance of domain, $\tilde h:M\to \tilde M$ is a homeomorphism. 
		
		\begin{figure*}[h!]
			\centering
			\begin{subfigure}{.31\textwidth}
				\begin{tikzcd}
					M 
						\arrow[rr,"\phi"]
						\arrow[rd,"\pi"] 
					& 
					& h^*\tilde M 
						\arrow[ld,"h^*\widetilde{\pi}"]\\
					& B &
				\end{tikzcd}
				\caption{}
				\label{fig:lifth1}
			\end{subfigure}
			\quad
			\begin{subfigure}{.31\textwidth}
				\begin{tikzcd}
					h^*\tilde M \arrow[r,"\proj_2"] \arrow[d,"h^*\widetilde{\pi}"]
					& \tilde M \arrow[d,"\widetilde{\pi}"]\\
					B \arrow[r,"h"]
					&B
				\end{tikzcd}
				\caption{}
				\label{fig:lifth2}
			\end{subfigure}
			\quad 
			\begin{subfigure}{.31\textwidth}
				\begin{tikzcd}
					M 	
						\arrow[r,"\phi"]
						\arrow[rd,"\pi"]
					& h^*\tilde M
						\arrow[r, "\proj_2"]
						\arrow[d,"h^*\widetilde{\pi}"]
					&\tilde M
						\arrow[d, "\widetilde{\pi}"]\\
					& B
						\arrow[r,"h"]
					&B
				\end{tikzcd}
				\caption{}
				\label{fig:lifth3}
			\end{subfigure}
			\caption{}
			\label{fig:lifth}
		\end{figure*}
		
	\end{proof}

\section{Lifting the Anosov automorphism on the leaf space to a partially hyperbolic system}\label{sect: lifting Anosov automorphism}

	\begin{lemma}\label{lem:liftingA}
		Let $F,E_0,E_1, B$ be closed Riemannian manifolds. Assume that $p_0:E_0\to B$ and $p_1:E_1\to B$ are $C^k$ $F$-bundles with structure group $H$, where $H$ is a finite-dimensional Lie group with smooth universal bundle, and that the left action of $H$ on $F$ is $C^k$, and that $H$ acts on $F$ by isometries. Suppose that $\theta:E_0\to E_1$ is a (continuous) isomorphism of $E_0$ and $E_1$ as $F$-bundles with structure group $H$ over $B$. Then, there is a $C^k$ isomorphism $\alpha:E_0\to E_1$ that is an isometry on fibers.
	\end{lemma}
		
	\begin{proof}
		We begin by constructing principal $H$-bundles with the same transition data as $E_0$ and $E_1$ using Lemma \ref{lem:fbconstruction}. We'll call these $q_0:Q_0\to B$ and $q_1:Q_1\to B$. Let $f_0:B\to BH$ and $f_1:B\to BH$ be classifying maps for $Q_0$ and $Q_1$. Note that since $E_0$ and $E_1$ (and therefore $Q_0$ and $Q_1$) are $C^k$ bundles, the maps $f_0$ and $f_1$ are $C^k$. 
		
		Since $E_0$ and $E_1$ are isomorphic as continuous bundles, we get that there is a homotopy $f:B\times [0,1]\to BH$ from $f_0$ to $f_1$. Since $f|_{B\times \{0,1\}}$ is $C^k$, then by the Whitney Approximation Theorem (See \cite{Lee} Theorem 6.26), $f_t$ is homotopic to a $C^k$ map $\overline{f}:B\times [0,1]\to BH$ relative to $B\times \{0,1\}$. So, we have a $C^k$ homotopy $\overline{f}:B\times [0,1]\to BH$ from $\overline{f_0}=f_0$ to $\overline{f_1}=f_1$. If the classifying maps of two $C^k$ principal bundles bundles are homotopic via a $C^k$ homotopy, then the bundles are isomorphic as $C^k$ bundles \cite[Chapter 4.9]{Husemoller}\footnote{The argument in \cite[Chapter 4.9]{Husemoller} is only given for continuous bundles, but works for $C^k$ bundles}. Thus, we get that the pullback bundles $f_0^*EH$ and $f_1^*EH$ are isomorphic as $C^k$ principal $H$-bundles over $B$. Since $f_0$ and $f_1$ are the classifying maps for $Q_0$ and $Q_1$ respectively, this means that $Q_0$ and $Q_1$ are isomorphic as $C^k$ principal $H$-bundles over $B$. Since $Q_0$ and $Q_1$ have the same transition functions as $E_0$ and $E_1$, we get that $E_0$ and $E_1$ are isomorphic as $C^k$ bundles with structure group $H$.

		From the definition of a $C^k$ isomorphism of $F$-bundles with structure group $H$, we see that this means that there is a $C^k$ isomorphism $\alpha:E_0\to E_1$, such that, if ${(U_{0,i}, \phi_{0,i})}$ and $\{(U_{1,j},\phi_{1,j})\}$ are trivializing atlases for $E_0$ and $E_1$ respectively, then there exists functions $d_{ij}:U_{0,i}\cap V_{1,j}\to H$ such that for $x\in U_{0,i}\cap U_{1,j}$ and $y\in F$, we get that $\phi_{1,j}\circ \alpha\circ \phi_{0,i}^{-1}(x,y)=(x,d_{ij}(x)\cdot y)$. Since $H$ acts on $F$ by isometries, we get that $\alpha$ is an isometry on fibers.
	\end{proof}
	
	\begin{corollary}\label{cor:liftingA}
		Let $F,M, B$ be closed Riemannian manifolds. Assume that $\pi:M\to B$ is a smooth $F$-bundle with structure group $H$, where $H$ is a finite-dimensional Lie group with smooth universal bundle, and that the left action of $H$ on $F$ is smooth, and that $H$ acts on $F$ by isometries. Suppose $A:B\to B$ is a smooth Anosov diffeomorphism, and suppose that $A$ lifts to a homeomorphism $\widehat{A}:M\to M$. Then $A$ lifts to a $C^\infty$ diffeomorphism $g:M\to M$ that is an isometry on fibers of $\pi:M\to M$.
	\end{corollary}
	
	\begin{proof}
		Since $A$ lifts to a homeomorphism $\widehat{A}: M\to M$, we can construct a (continuous) isomorphism $\theta: M\to A^*M$ of $M$ and $A^*M$ as $F$-bundles with structure group $H$ given by $\theta(z)=\left(\pi(z), \widehat{A}(z) \right)$ (See Figure \ref{fig:isomorphism and lift1}). By Lemma \ref{lem:liftingA}, there is a $C^\infty$ isomorphism $\alpha:M\to A^*M$ that is an isometry on fibers. We can then use $\alpha:M\to A^*M$ to define a $C^\infty$ diffeomorphism $g:M\to M$ by $g(z)=\proj_2\circ \alpha(z)$ (See Figure \ref{fig:isomorphism and lift2}). Since $\alpha$ is an isometry on fibers, we get that $g$ is an isometry on fibers.
		
		\begin{figure}[h!]
			\centering
			\begin{subfigure}{.4\textwidth}
				\begin{tikzcd}
					M
						\arrow[rrd,"\widehat{A}", bend left]
						\arrow[rd, "\theta"]
						\arrow[rdd,"\pi", bend right]
					& &
				\\
					& A^*M
						\arrow[r,"\proj_2"]
						\arrow[d,"A^*\pi"]
					& M
						\arrow[d,"\pi"]
				\\
					& B 
						\arrow[r,"A"]
					& B
				\end{tikzcd}
				\caption{}
				\label{fig:isomorphism and lift1}
			\end{subfigure}
			\qquad
			\begin{subfigure}{.4\textwidth}
				\begin{tikzcd}
					M
						\arrow[rrd,"g", bend left]
						\arrow[rd, "\alpha"]
						\arrow[rdd,"\pi", bend right]
					& &
				\\
					& A^*M
						\arrow[r,"\proj_2"]
						\arrow[d,"A^*\pi"]
					& M
						\arrow[d,"\pi"]
				\\
					& B 
						\arrow[r,"A"]
					& B
				\end{tikzcd}
				\caption{}
				\label{fig:isomorphism and lift2}
			\end{subfigure}
			\caption{}
			\label{fig:isomorphism and lift}
		\end{figure}
	\end{proof}
	
	\begin{prop}\label{prop: lift of Anosov isometry on fibers is ph}
		Let $F,M,$ and $B$ be closed Riemannian manifolds. Assume that $\pi:M\to B$ is a smooth $F$-bundle and that $f:B\to B$ is an Anosov diffeomorphism. If $g:M\to M$ is a diffeomorphism that is a lift of $f$ and such that $g$ is an isometry on fibers of $\pi:M\to B$, then $g:M\to M$ is partially hyperbolic.
	\end{prop}
	
	\begin{proof}
		We begin by constructing a Riemannian metric on $M$, with respect to which $g$ is partially hyperbolic. This construction has three ingredients:
		
		\begin{itemize}
		\item 
			A smooth family $\langle \cdot,\cdot\rangle^F_x$ of Riemannian metrics on the fibers $\pi^{-1}(x)$ such that $Dg:T\pi^{-1}(x)\to T\pi^{-1}(f(x))$ is an isometry for all $x\in B$. (Such a family exists because $g$ is an isometry on fibers of $\pi:M\to B$.)
		\item
			A Riemannian metric $\langle \cdot, \cdot \rangle^B$ on $B$ that is adapted to the Anosov diffeomorphism $f:B\to B$.
		\item
			An Ehresmann connection $H$ on $M$, i.e. $H$ is a smooth subbundle of the tangent bundle $TM$ such that for all $p\in M$, $T_pM=H_p\oplus \ker(D_p\pi)$. Note that from the definition of an Ehresmann connection, we know that $D_p\pi|_{H_p}:H_p\subset T_pM \to T_{\pi(p)}B$ is an isomorphism and the map $p\mapsto H_p$ is smooth.
		\end{itemize}
		
		We define a Riemannian metric $\langle \cdot, \cdot\rangle$ on $M$ by letting for all $p\in M$,
		
		\begin{itemize}
		\item 
			$\langle v,v'\rangle =\left\langle D_p\pi(v), D_p\pi(v')\right\rangle^B$ for $v,v'\in H_p$,
		\item
			$\langle v,v'\rangle = \langle v,v'\rangle_{\pi(p)}^F$ for $v,v'\in \ker(D_p\pi)$, and 
		\item 
			$H_p\perp \ker(D_p \pi)$.
		\end{itemize}
		
		Now, we need to show that $g$ is partially hyperbolic with respect to the metric $\langle \cdot, \cdot \rangle$. To do this, we construct a dominated splitting $TM=E^s\oplus E^c\oplus E^u$ such that $g$ is uniformly contracting on $E^s$ and uniformly expanding on $E^u$. We begin by letting $E^c=\ker(D\pi)$. 
		
		Next, we construct the unstable bundle $E^u$ using a graph transform argument. We begin by lifting the unstable bundle $E^u_f\subset TB$ for the Anosov diffeomorphism $f:B\to B$ to the bundle $\hat E^u\subset TM$ given by $\hat E^u_p:= H_p\cap D_p\pi^{-1}(E^u_f(\pi(p)))$ for $p\in M$. Note that $D_p\pi^{-1}(E^u_p(\pi(p)))=\hat E^u_p \oplus \ker(D_p\pi)$, and that $Dg$ preserves $\hat E^u\oplus E^c$ because $Df$ preserves $E^u_f$ and $g$ covers $f$ (so $D_{\pi(p)}f\circ D_p\pi=D_{g(p)}\pi\circ D_pg$). 
		
		Let 
			$$
				\Sigma=\left\{\sigma:\hat E^u\to E^c: \ \sigma \text{ is fiber preserving over } id, \text{ and } \sigma_p:\hat E^u(p)\to E^c(p) \text{ is linear } \forall p\in M \right\}
			$$
		We put the norm $\|\cdot\|_\Sigma$ on $\Sigma$ given by 
			$$
				\|\sigma\|_\Sigma=\sup_{p\in M} \|\sigma_p\|
			$$
		where $\|\sigma_p\|$ is the operator norm. Note that this norm makes $\Sigma$ a Banach space. 
		
		Now, we want to define a map $\Gamma:\Sigma \to \Sigma$, called the linear graph transform covering $g$, so that $D_pg(\graph(\sigma_p))=\graph(\Gamma(\sigma_p))$. We now give $\Gamma:\Sigma\to \Sigma$ explicitly. Since by assumption, $Dg$ preserves $E^c=\ker(D\pi)$ (and $Dg$ preserves $\hat E^u\oplus E^c$), we can write for each $p\in M$, 
			$$
				D_pg=\begin{pmatrix}
					A_p & 0\\ C_p & K_p
				\end{pmatrix}:\hat E^u(p)\oplus E^c(p) \to \hat E^u(g(p))\oplus E^c(g(p)),
			$$
		where 
			$$
				A_p:\hat E^u(p)\to \hat E^u(g(p)), \qquad C_p:\hat E^u(p)\to E^c(g(p)), \qquad K_p:E^c(p)\to E^c(g(p))
			$$
		are all linear. Also note that since $D_pg$ is invertible, both $A_p$ and $K_p$ are invertible.
		
		Note that we can write a point in the graph of $\sigma_p$ as $(v,\sigma_pv)\in \graph(\sigma_p)\subset \hat E^u(p)\oplus E^c(p)$. Applying $D_pg$ to this point gives us 
			$$
				D_pg(v,\sigma_p v)
				=
				\begin{pmatrix}
					A_p&0\\C_p&K_p
				\end{pmatrix}
				\begin{pmatrix}
					v \\ \sigma_pv
				\end{pmatrix}
				=
				\begin{pmatrix}
					A_pv\\ C_p v +K_p\sigma_p v
				\end{pmatrix}
			$$
		So, we can write
			$$
				D_pg(\graph(\sigma_p))=
				\left\{ 
					\begin{pmatrix}
						A_pv\\ C_p v +K_p\sigma_p v
					\end{pmatrix} 
					: \ v\in \hat E^u(p)
				\right\}
				=
				\left\{ 
					\begin{pmatrix}
						w\\ \left(C_p  +K_p\sigma_p \right)\circ A_p^{-1}w
					\end{pmatrix} 
					: \ w\in \hat E^u(g(p))
				\right\}	
			$$
		by reparametrizing. Thus, the requirement that $D_pg(\graph(\sigma_p))=\graph(\Gamma(\sigma_p))$ is equivalent to saying that 
			$$
					\begin{pmatrix}
						w\\ \left(C_p  +K_p\sigma_p \right)\circ A_p^{-1}w
					\end{pmatrix} 		
				=
					\begin{pmatrix}
						w \\ \Gamma\sigma_p w
					\end{pmatrix}		
			$$
		for all $w\in \hat E^u(g(p))$. This gives us an equation for $\Gamma$ in terms of $C, \ K,$ and $A$:
			$$
				\Gamma\sigma_p =\left(C_p  +K_p\sigma_p \right)\circ A_p^{-1}:\hat E^u(g(p))\to E^c(g(p))
			$$
		for all $\sigma\in \Sigma$, $p\in M$, and $v\in \hat E^u(p)$. Omitting base points, we get that 
			$$
				\Gamma \sigma = \left(C  +K\sigma \right)\circ A^{-1}
			$$
		
		Now, our goal is to find an invariant section for $\Gamma$ (the graph of which we will then show is the unstable bundle $E^u$ for $g$). To do this, it suffices to show that $\Gamma:\Sigma\to \Sigma$ is a contraction. Take $\sigma, \sigma'\in \Sigma$. For $p\in M$, we have 
			\begin{align}
					\|\Gamma \sigma_p - \Gamma \sigma_p'\|
				&= 
					\left\| \left( C_p+K_p \sigma_p \right)\circ A_p^{-1} - \left( C_p+K_p \sigma_p' \right)\circ A_p^{-1} \right\|\nonumber\\
				&= 
					\left\| \left( C_p+K_p \sigma_p-C_p-K_p\sigma'_p \right)\circ A_p^{-1}\right\|\nonumber\\
				&=
					\left\| \left(K_p \sigma_p-K_p\sigma'_p \right)\circ A_p^{-1}\right\|\nonumber\\
				&=
					\left\| K_p \circ \left(\sigma_p-\sigma'_p  \right)\circ A_p^{-1}\right\|\nonumber\\
				&\leq 
					\|K_p\| \|\sigma_p-\sigma'_p\|\|A_p^{-1}\| \label{eq:graph transform contraction}
			\end{align}
		We now bound $\|K_p\|$ and $\|A_p^{-1}\|$. Since $g$ is an isometry on fibers of $\pi:M\to B$ and $E^c=\ker(D\pi)$, we get have that $D_pg|_{E^c(p)}=K_p:E^c(p)\to E^c(g(p))$ is an isometry. Thus, $\|K_p\|=1$. Now, we bound $\|A_p^{-1}\|$ by relating the norm of $A$ to the norm of $Df$ on $E^u_f$. Since $E^u_f$ is the unstable bundle for the Anosov diffeomorphism $f:B\to B$ and the norm $\|\cdot \|^B$ is adapted to $f$, we know that there exists a constant $\lambda>1$ such that for all $w\in E^u_f$, $\|Df(w)\|^B\geq \lambda \|w\|^B$. Take $v\in \hat E^u(p)$. Since $D_p\pi(v)\in E^u_f(\pi(p))$, we therefore have that
			$$
				\|D_{\pi(p)}f(D_p\pi(v))\|^B\geq \lambda \|D_p(v)\|^B.
			$$
		Since $g$ covers $f$, we see that $D_{\pi(p)}f( D_p\pi(v))=D_{g(p)}\pi(D_pg(v))$ This along with the fact that $E^c=\ker(D\pi)$ and $C_p(v)\in E^c(g(p))$ gives that 
			$$
					D_{\pi(p)}f(D_p\pi(v))
				=
					D_{g(p)}\pi(D_pg(v))
				=
					D_{g(p)}\pi(A_p(v)+C_p(v))
				=
					D_{g(p)}\pi(A_p(v))
			$$
		Thus,
			$$
					\lambda \|D_p(v)\|^B
				\leq 
					\|D_{\pi(p)}f(D_p\pi(v))\|^B
				=
					\| D_{g(p)}\pi(A_p(v)) \|^B
			$$
		Finally, note that since $v, A_p(v)\in \hat E^u(P)\subset H_p$, by our definition of the norm on $M$, we get that 
			$$
				\|D_p\pi(v)\|^B=\|v\| \qquad \text{ and } \qquad \| D_{g(p)}\pi(A_p(v)) \|^B=\|A_p(v)\|.
			$$
		We've therefore shown that $\lambda\|v\|\leq \|A_p(v)\|$, which implies that $\|A_p^{-1}\|\leq \lambda^{-1}$. 
		
		Combining our estimates for the norms of $\|K_p\|$ and $\|A_p^{-1}\|$ with \eqref{eq:graph transform contraction} gives that 
			$$
				\|\Gamma \sigma_p - \Gamma \sigma_p'\|\leq \lambda^{-1}\|\sigma_p-\sigma'_p\|.
			$$
		We have therefore shown that $\Gamma$ is a contraction map. Then, by the contraction mapping principle, we get a $\Gamma$-invariant section $\sigma^u\in \Sigma$. We now define a bundle $E^u\subset TM$ by letting $E^u(p):=\graph(\sigma^u_p)$. Note that $E^u$ is $Dg$ invariant since $\sigma^u$ is $\Gamma$ invariant and $D_pg(\graph(\sigma_p))=\graph(\Gamma(\sigma_p))=\graph(\sigma_{g(p)})$.
		
		The construction of the bundle $E^s$ is analogous. 
		
		We now have a $Dg$-invariant splitting, $TM=E^s\oplus E^c\oplus E^u$. We now need to show that this splitting is partially hyperbolic. To do this, we construct a new metric $\langle \cdot, \cdot \rangle'$ on $M$ by letting for all $p\in M$,
			\begin{itemize}
			\item $\langle v,v'\rangle' = \langle D_p\pi(v),D_p\pi(v')\rangle^B$ for $v,v'\in E^s$,
			\item $\langle v,v'\rangle' = \langle D_p\pi(v),D_p\pi(v')\rangle^B$ for $v,v'\in E^u$,
			\item $\langle v,v'\rangle' = \langle v,v' \rangle^F_{\pi(p)}$ for $v,v'\in E^c$, and 
			\item $E^s$, $E^c$, and $E^u$ be pairwise orthogonal.
			\end{itemize}
		From our construction of $E^s$ and $E^u$, we get that $Dg$ is uniformly expanding on $E^u$ and uniformly contracting on $E^c$ with respect to this new metric. Finally, the splitting is dominated because $Dg$ restricted to $E^c$ is an isometry.  
	\end{proof}

\section{Proof of Theorem \ref{thm:smoothmodel}}\label{sect:proof of main thm}

	First, we recall our setup. Let $\hat{f}:M\to M$ be a fibered partially hyperbolic system with quotient a nilmanifold $B$, $C^1$ fibers $F$ (where $F$ is a closed manifold), and structure group $G\subset \Diff^1(F)$. Suppose that there exists a Riemannian metric on F and a subgroup $I\subset \Isom(F)\cap G$ such that the inclusion $I \hookrightarrow G$ is a homotopy equivalence. 
	
	The diffeomorphism $\hat{f}:M\to M$ descends to a homeomorphism $f:B\to B$. Our first step is to construct a conjugacy $h:B\to B$ between $f$ and a hyperbolic nilmanifold automorphism $A:B\to B$. This will follow immediately from Theorem \ref{thm:HiraideNilmanifolds} if we can show that $f:B\to B$ is an Anosov homeomorphism. To see why the homeomorphism  $f:B\to B$ is Anosov, we first observe that $\hat f$ admits an invariant center foliation $\F^c$ whose leaves are the level sets of $\pi$ and that leaves of $\F^c$ are compact and have trivial holonomy.
		\footnote{For the definition of holonomy, see \cite[Chapter 2]{CandelConlon}. The fact that the leaves of $\F^c$ have trivial holonomy follows immediately from the definition of a fibered partially hyperbolic system and the definition of holonomy.} 
	Thus, by the following result of Bohnet and Bonatti, $f:B\to B$ is an \textit{Anosov homeomorphism}.
		
		\begin{lemma}[{\cite[Theorem 2, Proposition 4.20]{BohnetBonatti}}] \label{lem:BohnetBonatti}
			If $f:M\to M$ is a partially hyperbolic diffeomorphism with an invariant center foliation $\F^c$ with compact leaves and without holonomy, then the homeomorphism $F:M/\F^c \to M/\F^c$ induced by $f$ on the quotient is an Anosov homeomorphism.
		\end{lemma}		
	
	Now, we can apply Theorem \ref{thm:HiraideNilmanifolds} to get that there exists a hyperbolic nilmanifold automorphism $A:B\to B$ and a homeomorphism $h:B\to B$ that is homotopic to the identity such that $A\circ h=h\circ \hat{f}$. 
	
	The next step in the proof is to construct a smooth $F$-bundle $\widehat{\pi}:\widehat M\to B$ that is isomorphic to the original bundle $\pi:M\to B$ and such that the structure group of $\widehat{\pi}:\widehat M \to B$ is $\Isom(F)$. To do this, we first construct a principal $G$ bundle $p:P\to B$ with the same transition functions as $\pi:M\to B$. Since the inclusion of $I\hookrightarrow G$ is a homotopy inclusion, by Lemma \ref{lem:reduction-structure-group} there exists a continuous principal $I$-bundle $q':Q'\to B$ that has transition functions cohomologous to those of $p:P\to B$. Since $I\subset \Isom(F)$, we can construct (Lemma \ref{lem:fbconstruction}) a continuous principal $\Isom(F)$-bundle $q:Q\to B$ with the same transition data as $q':Q'\to B$. 
	
	Now, we find a smooth principal $\Isom(F)$ bundle $\hat q:\hat Q\to B$ that is isomorphic to $q:Q\to B$. This follows immediately from the following lemma along with the fact that $\Isom(F)$ is a locally Euclidean Lie group \cite{MyersSteenrod}.
	
		\begin{lemma}[{\cite{MullerWockel}}]\label{lem:MullerWorkel}
			Let $K$ be a Lie group modeled on a locally convex space. Every principal $K$ bundle over a closed manifold is isomorphic to a smooth principal $K$ bundle.
		\end{lemma}
	
	Now, we use the fiber bundle construction theorem (Lemma \ref{lem:fbconstruction}) to construct a smooth $F$-bundle $\widehat \pi:\widehat M\to B$ with the same transition functions as $\hat q: \hat Q\to B$. Since $\hat q:\hat Q\to B$ has transition functions that are cohomologous to those of the original bundle $\pi:M\to B$, we get that the bundle $\widehat \pi: \widehat M\to B$ is isomorphic to the original bundle $\pi:M\to B$. 
	
	Next, we lift the conjugacy $h:B\to B$ to a homeomorphism $\hat h: M\to \widehat M$. This follows immediately from Lemma \ref{lem:liftingconjugacy}. 
	
	Finally, we lift the hyperbolic nilmanifold automorphism $A:B\to B$ to a partially hyperbolic diffeomorphism $g:\widehat M \to \widehat M$. This follows immediately from Corollary \ref{cor:liftingA} and Proposition \ref{prop: lift of Anosov isometry on fibers is ph}. To see why we can apply Corollary \ref{cor:liftingA} here, we first note that the structure group of $\widehat \pi: \widehat M \to B$ is $\Isom(F)$, which is a finite dimensional compact Lie group \cite{MyersSteenrod}. This implies that $\Isom(F)$ has a smooth universal bundle \cite[Lemma I.12]{MullerWockel}.  This completes the proof of Theorem \ref{thm:smoothmodel}.

	\begin{remark}\label{remark:assumption on structure group}
		The fact that the $F$-bundle $\widehat M$ has structure group $\Isom(F)$ is solely used to guarantee that the lift $g:\widehat M\to \widehat M$ of $A:B\to B$ is partially hyperbolic. Without this fact, the arguments given would allow us to lift $A:B\to B$ to a homeomorphism, but we would not be able to guarantee that the lift would be a partially hyperbolic diffeomorphism.
		
		This is the only reason that the assumption that that there exists a Riemannian metric on $F$ and a subgroup $I\subset \Isom(F)\cap G$ such that the inclusion $I\hookrightarrow G$ is a homotopy equivalence is necessary in the proof. Without this assumption, we would be able to get the conjugacy $h:B\to B$ between $f$ and $A$, and we would be able to construct a smooth $F$-bundle $\widehat M$ over $B$ that is isomorphic to the original $F$-bundle $M$.\footnote{To construct $\widehat M$, we would first use Lemma \ref{lem:reduction-structure-group} and Proposition \ref{prop:InclusionDiffeoGrps} to get a $F$-bundle with structure group $\Diff^\infty(F)$ that is isomorphic to $M$. We then would apply Lemma \ref{lem:MullerWorkel} with $K=\Diff^\infty(F)$, which would give us $\widehat M$.} However, the structure group of $\widehat M$ would only be $\Diff^\infty(F)$, not $\Isom(F)$. 
		 
		Finding a way to lift $A:B\to B$ to a partially hyperbolic diffeomorphism $g:\widehat M\to \widehat M$ without requiring that the structure group of $\widehat M$ be $\Isom(F)$  or be trivial is a question for further research.
	\end{remark}
	
	Proposition \ref{prop: trivial case main thm} is an example of a case where we can overcome the difficulty lifting $A:B\to B$ to a partially hyperbolic diffeomorphism that is discussed in the above remark. 
	
	\begin{proof}[Proof of Proposition \ref{prop: trivial case main thm}]
		The setup of Proposition \ref{prop: trivial case main thm} is that we are given a fibered partially hyperbolic system $\hat f:M\to M$ with quotient a nilmanifold $B$ and $C^1$ fibers $F$. We assume that the $F$-bundle $M$ is trivial. Note that this means that the structure group of the bundle $\pi:M\to B$ is the trivial group. We can then proceed with an analogous argument to the one given in the proof of Theorem \ref{thm:smoothmodel} to get the conjugacy $h:B\to B$ between the Anosov homeomorphism $f:B\to B$ induced by $\hat f:M\to M$ and a hyperbolic nilmanifold automorphism $A:B\to B$ and to get a smooth bundle $\widehat \pi : \widehat M \to B$ with trivial structure group that is isomorphic to the original bundle $\pi: M\to B$. This means that identifying the smooth $F$-bundle $\widehat \pi: \widehat M \to B$ with the smooth bundle $\proj_1: B\times F \to B$ (that is projection onto the first coordinate), we can smoothly lift $A:B\to B$ to the fibered partially hyperbolic diffeomorphism $g:\widehat M\cong B\times F\to \widehat M\cong B\times F$ given by $g: (x,y)\mapsto (Ax,y)$ for $(x,y)\in B\times F$. 
	\end{proof}

\section{Corollaries of the Theorem \ref{thm:smoothmodel}}\label{sect:corollaries}

	We now explain how Corollaries \ref{cor:StructureGroupDiff} and \ref{cor:StructureGroupConnectedComponent} follow from Theorem \ref{thm:smoothmodel}.
	
	To prove Corollary \ref{cor:StructureGroupDiff}, we apply Theorem \ref{thm:smoothmodel} with $G=\Diff^1(F)$ and with $I=\Isom(F)$. To do this, we just need to show that the inclusion $\Isom(F)\hookrightarrow \Diff^1(F)$ is a homotopy equivalence when $F=S^n$ for $n=1,2,3$ and for $F$ a hyperbolic 3-manifold. In fact showing that the inclusion $\Isom(F)\hookrightarrow \Diff^1(F)$ is a homotopy equivalence is equivalent to showing that the inclusion $\Isom(F)\hookrightarrow \Diff^\infty(F)$ is a homotopy equivalence by Proposition \ref{prop:InclusionDiffeoGrps}.
		\begin{enumerate}
		\item 
			When $F=S^1$ it is a standard fact that $\Diff^\infty(S^1)$ deformation-retracts to $O(2)=\Isom(S^1)$. When $F=S^2$, Smale \cite{SmaleDiffofS2} proved that the inclusion $\Isom(S^2)\hookrightarrow \Diff^\infty(S^2)$ is a homotopy equivalence. Hatcher \cite{HatcherSmaleConj} proved this for $S^3$.
		\item 
			When $F$ is a hyperbolic $3$-manifold, Gabai \cite{Gabai} proved that the inclusion $\Isom(F)\hookrightarrow \Diff^\infty(F)$ is a homotopy equivalence.
		\end{enumerate}
		
	\begin{remark}
		When $n\geq 4$, the inclusion $\Isom(S^n)\hookrightarrow \Diff(S^n)$ is not a homotopy equivalence. This was proved for $n=4$ in \cite{WatanabeSmaleConjS4}. To see that $\Isom(S^n)\hookrightarrow \Diff(S^n)$ is not a homotopy equivalence for $n\geq 5$, first note that this statement is equivalent to the statement that $\Diff(D^n \ \text{rel } \partial D^n)$ is contractible \cite[Appendix]{HatcherSmaleConj}. 
		
		One way to see that $\Diff(D^n \ \text{rel } \partial D^n)$ is not contractible for many $n$ is to use the fact that $\pi_0(\Diff(D^n \ \text{rel } \partial D^n))\cong \Theta_{n+1}$ for $n\geq 5$, where $\Theta_{n+1}$ is the group of exotic $(n+1)$-spheres \cite{HatcherDiffeoGrps} \cite{smale_generalized_1961} \cite{cerf_stratification_1970}. For example, this along with the fact that there exist exotic 7-spheres \cite{milnor_manifolds_1956} implies that $\Isom(S^6)\hookrightarrow \Diff(S^6)$. 
		
		To prove that $\Diff(D^n \ \text{rel } \partial D^n)$ is not contractible for $n=5$, we use the fact that the map $\pi_1(\Diff(D^n \text{ rel } \partial D^n))\to \pi_0(\Diff(D^{n+1} \text{ rel } \partial D^{n+1}))$ is surjective for $n\geq 5$ \cite{cerf_stratification_1970}. Thus, since $\pi_0(\Diff(D^{6} \text{ rel } \partial D^{6}))\neq 0$, we get that $\pi_1(\Diff(D^5 \text{ rel } \partial D^5))\neq 0$, so $\Diff(D^5 \text{ rel } \partial D^5)$ is not contractible.
		
		For $n\geq 7$, the fact that $\Diff(D^n \ \text{rel } \partial D^n)$ is not contractible is proved in \cite{crowley_gromoll_2013}. 
		
		This means that we cannot apply Theorem \ref{thm:smoothmodel} as above to get an analogous version of Corollary \ref{cor:StructureGroupDiff} for $S^n$, $n\geq 4$. 
	\end{remark}
	
	Now, we prove Corollary \ref{cor:StructureGroupConnectedComponent}. In Corollary \ref{cor:StructureGroupConnectedComponent}, the structure group of $M$ is $G=\Diff_0^1(F)$. We prove each case of Corollary \ref{cor:StructureGroupConnectedComponent} separately using the same strategy: we apply Theorem \ref{thm:smoothmodel} by finding a subgroup $I\subset \Isom(F)\cap \Diff^1_0(F)$ such that the inclusion $I\hookrightarrow \Diff^1_0(F)$ is a homotopy equivalence.
		\begin{enumerate}
		\item 
			When $F=\T^n$ for $n=2,3$, we choose $I=\T^n$, where $\T^n$ acts on itself by Euclidean isometry via translation. When $F=\T^2$, the inclusion $\mathbb{T}^2\hookrightarrow \Diff^\infty_{0}(\mathbb{T}^2)$ is a homotopy equivalence \cite[Section 4.1.5]{Morita}. When $F=\T^3$, $\mathbb{T}^3\hookrightarrow \Diff^\infty_{0}(\mathbb{T}^3)$ is a homotopy equivalence \cite{Hatcher3manifoldsUpdate} \cite{Waldhausen}.
		\item 
			When $F$ is a hyperbolic surface, then $\Diff^+_0(F)$ is contractible \cite[Thereom 1.14]{FarbMargalit}, which means that the hypothesis of Theorem \ref{thm:smoothmodel} holds for $I$ the trivial subgroup.
		\item
			When $F$ is a Haken manifold we consider three cases:
			
			\begin{itemize}
			\item 
				When $F$ is not a Seifert manifold with coherently orientable fibers, then the components of $\Diff(F)$ are contractible \cite{Hatcher3manifolds}, \cite{Hatcher3manifoldsUpdate}, \cite[Section 1.3]{DiffeosElliptic3Manifolds}, which means that the hypothesis of Theorem \ref{thm:smoothmodel} holds when $I$ is the trivial subgroup.
			\item 
				When $F$ is a Seifert manifold with coherently oriented fibers that is not $\T^3$, we take $I=S^1$, where $S^1$ acts on $F$ by rotating circle fibers of the Seifert fiber bundle structure. This satisfies the hypothesis of Theorem \ref{thm:smoothmodel} because the inclusion $S^1\hookrightarrow \Diff_0(F)$ is a homotopy equivalence \cite{Hatcher3manifolds}, \cite{Hatcher3manifoldsUpdate}, \cite[Section 1.3]{DiffeosElliptic3Manifolds}.
			\item
				When $F=\mathbb{T}^3$, we dealt with this case in \eqref{cor:fibers torus}. 
			\end{itemize}
			
		\end{enumerate}
		
\section{Appendix: Anosov Homeomorphisms}\label{sec:HiraideGeneralization}

	The goal of this appendix is to provide a proof of Theorem \ref{thm:HiraideNilmanifolds}. This result was initially proved by Sumi \cite{Sumi}, and we provide a proof which follows the same structure as that of Sumi's for the sake of completeness. Theorem \ref{thm:HiraideNilmanifolds} extends the following result of Franks and Manning to Anosov homeomorphisms of nilmanifolds.
	
		\begin{theorem}[\cite{Franks_Anosov_tori},\cite{Manning_conjugacy}]\label{thm:FranksManning}
			An Anosov diffeomorphism of a nilmanifold is topologically conjugate to a hyperbolic nilmanifold automorphism.
		\end{theorem}
		
	Theorem \ref{thm:HiraideNilmanifolds} generalizes the following result of Hiraide from tori to nilmanifolds.
	 	\begin{theorem}[\cite{Hiraide_tori}]\label{thm:Hiraide}
	 		An Anosov homeomorphism of a torus is topologically conjugate to a hyperbolic toral automorphism.
	 	\end{theorem}
	 The proof of Theorem \ref{thm:HiraideNilmanifolds} follows the same basic structure as Sumi's proof, and also the same structure as Hiraide's proof with some modifications to account for being on a nilmanifold instead of a torus.

\subsection{Notation and Preliminaries for the Proof of Theorem \ref{thm:HiraideNilmanifolds}}\label{subsec:preliminaries}

	This section recalls several properties of Anosov homeomorphisms that will be necessary to the proof of Theorem \ref{thm:HiraideNilmanifolds}. 
	
	We assume in the following that $M$ is a connected, closed $n$-dimensional Riemannian manifold. Let $d$ be the distance function on $M$ induced by the Riemannian metric.

\subsubsection{Generalized foliations for Anosov homeomorphisms}
	
	In this section, we describe an analogue of the stable manifold theorem for Anosov homeomorphisms that is due to Hiraide \cite{Hiraide_tori}. Let $f:M\to M$ be a homeomorphism. For each $x\in M$, we define the \textit{stable set} (resp. \textit{unstable set}) of $f$ at $x$ as
		\begin{align*}
				W^s(x) &:= \left\{ y\in X: d(f^n(x),f^n(y)) \overset{n\to \infty}{\longrightarrow} 0 \right\}\\
				&\left( \text{resp. } 
						W^u(x) := \left\{ y\in X: d(f^{-n}(x),f^{-n}(y))\overset{n\to \infty}{\longrightarrow} 0 \right\}\right)
		\end{align*}
	The collection of stable (resp. unstable) sets for $f$, which we'll denote by $\F^s_f$ (resp. $\F_f^u$), gives a $f$-invariant decomposition of $M$. The stable manifold theorem states that when $f$ is an Anosov diffeomorphism, these collections form foliations. When $f$ is an Anosov homeomorphism, we get the following analogue.
	
		\begin{theorem}[{\cite[Proposition A]{Hiraide_tori}}]\label{propA Hiraide}
			If $f:M\to M$ is an Anosov homeomorphism of the closed manifold $M$, then the collections
				$$
					\F_f^\sigma=\{ W^\sigma(x): \ x\in M \}, \quad \sigma\in \{s,u\}
				$$
			are transverse generalized foliations of $M$.
		\end{theorem}
		
		\begin{remark}
			When $f$ is the projection of a fibered partially hyperbolic diffeomorphism $\hat{f}$, the transverse generalized foliations $W^s$ and $W^u$ are, in fact, foliations. They are the projections of the foliations $\widehat{W^s}$ and $\widehat{W^u}$ for $\hat{f}$.
		\end{remark}
		
	Generalized foliations are a generalization of foliations given by weakening the condition that the leaves be manifolds. More precisely, a collection $\F$ of subsets of $M$ is a \textit{generalized foliation} of $M$ if the following properties hold:
		\begin{enumerate}
			\item
				$\F$ is a partition of $M$.
			\item
				Each $L\in \F$ (called a \textit{leaf}) is path-connected.
			\item
				For each $x\in M$, there exist
					\begin{itemize}
					\item 
						nontrivial, connected subsets $D_x,K_x\subset M$ with $D_x\cap K_x=\{x\}$,
					\item 
						a connected, open neighborhood $N_x\subset M$ of $x$,
					\item
						a homeomorphism $\phi_x:D_x\times K_x\to N_x$ (called \textit{local coordinates around $x$})
					\end{itemize}
				such that 
					\begin{enumerate}
						\item $\phi_x(x,x)=x$,
						\item $\phi_x(y,x)=y \ \forall y\in D_x$ and $\phi_x(x,z)=z \ \forall z\in K_x$,
						\item For any $L\in \F$, there is at most a countable set $B\subset K_x$ such that $N_x\cap L=\phi_x(D_x\times B)$.
					\end{enumerate}
		\end{enumerate}
	Two generalized foliations $\F$ and $\F'$ on $M$ are \textit{transverse} if, for each $x\in M$, there exist
			\begin{itemize}
				\item 
					nontrivial, connected subsets $D_x,D'_x\subset M$ with $D_x\cap D_x'=\{x\}$,
				\item	
					a connected, open neighborhood $N_x$ of $x$ (called \textit{a coordinate domain}),
				\item
					a homeomorphism $\phi_x:D_x\times D'_x\to N_x$ (called a \textit{canonical coordinate chart (around $x$)}),
			\end{itemize}
		such that 
			\begin{enumerate}[(a)]
				\item $\phi_x(x,x)=x$,
				\item $\phi_x(y,x)=y, \ \forall y\in D_x$, and $\phi_x(x,z)=z \ \forall z\in D'_x$,
				\item for any $L\in \F$, there is at most a countable set $B'\subset D'_x$ such that $N_x\cap L=\phi_x(D_x\times B')$,
				\item for any $L'\in \F'$, there is at most a countable set $B\subset D_x$ such that $N_x\cap L'=\phi_x(B\times D'_x)$.
			\end{enumerate}
	Note that the sole difference between the definitions of a foliation and of a generalized foliation is we don't require the sets $D_x$ and $K_x$ to be manifolds in the definition of a generalized foliation. (If $D_x$ and $K_x$ are manifolds for all $x\in M$, then a generalized foliation $\F$ is, in fact, a topological foliation of $M$.) While the sets $D_x$ and $K_x$ may fail to be manifolds, the fact that their product, $D_x\times K_x$, is a manifold  significantly restricts the ways in which $D_x$ and $K_x$ can fail to be manifolds. In other words, $D_x$ and $K_x$ (and therefore the leaves of $\F$), while not necessarily manifolds themselves, will behave like manifolds in many ways. In fact, the leaves of a generalized foliation are homology manifolds (also known as generalized manifolds). An homology manifold is a topological space that looks like a manifold under homology. This is stated precisely in the following proposition.
		
		\begin{prop}[{\cite[Lemma 4.2]{Hiraide_tori}}]\label{prop:gfoliation homology}
			Let $\F$ be a generalized foliation on a connected manifold without boundary. There exists $0<p<\dim(M)$ such that any leaf $L\in \F$ and $x\in L$, the relative homology groups, $H_*(L,L\setminus \{x\})$, are given by 
				$$
					H_i(L,L\setminus \{x\})=
					\begin{cases}
						\Z, & \text{ if } i=p\\
						0, & \text{ if } i\neq p
					\end{cases}.
				$$
		\end{prop}
	This proposition allows us to define a notion of dimension for generalized foliation. If $\F$ is a generalized foliation of $M$, then the integer $p$ from Proposition \ref{prop:gfoliation homology} is called the \textit{dimension} of $f$. Proposition \ref{prop:gfoliation homology} also allows us to define orientability for generalized foliations. A $p$-dimensional generalized foliation is said to be \textit{orientable} if there is a `locally consistent' choice of generators for the groups $H_p(L,L\setminus \{x\})$, $L\in \F$ and $x\in L$. For more details see \cite[Section 4]{Hiraide_tori}.

\subsubsection{Lifts of stable and unstable sets}\label{subsect:lifts}

	Much of Franks' and Manning's proofs of Theorem \ref{thm:FranksManning} take place using maps lifted to the universal cover. These arguments exploit the facts that Anosov diffeomorphisms lift to Anosov diffeomorphisms whose stable and unstable sets are lifts of the original stable and unstable manifolds. We'll now give versions of these facts for Anosov homeomorphisms, which will be used in our proof of Theorem \ref{thm:HiraideNilmanifolds}. 
	
	We begin with the following set-up. Let $M$ be a closed Riemannian manifold and let $p:\tilde M\to M$ be a smooth covering map for $M$. By lifting the Riemannian metric on $M$, we see that $\tilde M$ is a complete Riemannian manifold. 
	
	We can now generalize the previous results about lifts of Anosov diffeomorphisms to Anosov homeomorphisms. These generalizations are due to Hiraide. For details on them and their proofs, see \cite[Section 3]{Hiraide_tori}. Let $f:M\to M$ be an Anosov homeomorphism. The map $f$ lifts to a homeomorphism $\tilde f:\tilde M\to \tilde M$. Just as Anosov diffeomorphisms lift to Anosov diffeomorphisms, we observe that an Anosov homeomorphism lifts to an Anosov homeomorphism. 
	
	Next, we'll discuss the relationship between the stable and unstable sets of $f$ and $\tilde f$. For $\tilde x\in \tilde M$ and $\varepsilon>0$, we let $\tilde W_\varepsilon^s(\tilde x)$ and $\tilde W_\varepsilon^u(\tilde x)$ be the local stable and unstable sets of $\tilde f$ at $\tilde x$. 
		\begin{align*}
				\tilde W_\varepsilon^s(\tilde x)
			&= 
				\left\{ y\in \tilde M : \ d\left( \tilde f^n\left( \tilde x\right), \tilde f^n \left(y \right)\right)\leq \varepsilon, \ \forall n\geq 0 \right\}, 
			\\
				\tilde W_\varepsilon^u(\tilde x)
			&= 
				\left\{ y\in \tilde M : \ d\left( \tilde f^{-n}\left( \tilde x\right), \tilde f^{-n} \left(y \right)\right)\leq \varepsilon, \ \forall n\geq 0 \right\}.			
		\end{align*}
	Just as for an Anosov diffeomorphism, the stable and unstable sets for $\tilde f$ project down to the stable and unstable sets for $f$. In fact, locally this projection is an isometry. In addition, the collection of stable (resp.) unstable sets of $\tilde f$ forms a generalized foliation, denoted $\F_{\tilde f}^s$ (resp. $\F_{\tilde f}^u$), and that the stable and unstable generalized foliations for $\tilde f$ are transverse. 
	
\subsubsection{Indices of fixed points}

	Let $f:M\to M$ be an Anosov diffeomorphism. The index of $f$ at any fixed point $x$, denoted $\Ind_x(f)$, will be either $\pm 1$ since $d_xf:T_xM\to T_xM$ is hyperbolic. The sign of $\Ind_x(f)$ will depend on the orientation of the stable and unstable subspaces, $E^s_x$ and $E^u_x$, at $x$. Thus, if the unstable bundle, $E^u$, of $f$ is orientable  (which implies that the unstable foliation for $f$ is orientable), we can make the fixed point index globally constant, i.e. for all $x,x'\in \Fix(f)$, $\Ind_{x}(f)=\Ind_{x'}(f)$. This along with the Lefschetz fixed point theorem tells us that the absolute value of the Lefschetz number of $f$, denoted $L(f)$, is equal to the number of fixed points of $f$. This fact is relied upon in the proof of Theorem \ref{thm:FranksManning}.
	
	The purpose of this section is to give the following similar result about the fixed point index of an Anosov homeomorphism, which will allow us to use the Lefschetz number to count fixed points. Note that we can define the fixed point index for a fixed point of an Anosov homeomorphism because all fixed points of an Anosov homeomorphism are isolated by expansivity.
	
	\begin{prop}[\cite{Hiraide_tori}, Proposition B]\label{propB Hiraide}
		Let $f:M\to M$ be an Anosov homeomorphism of the closed manifold $M$. If the generalized unstable foliation $\F_f^u$ is orientable, then for sufficiently large $m$, all the fixed points of $f^m$ have the same index, which is either $1$ or $-1$. 
	\end{prop}
	
	Note that the assumption in this proposition (i.e. that the generalized unstable foliation be orientable) is analogous to the assumption we made in the Anosov case. For the definition of orientability for a generalized foliation, see \cite{Hiraide_tori}. The proof of Theorem \ref{propB Hiraide} can be found in \cite[Section 5]{Hiraide_tori}. 
	
\subsubsection{The spectral decomposition}

	The spectral decomposition is a useful tool for decomposing the non-wandering set of an Anosov diffeomorphism into smaller invariant sets. Recall that given a homeomorphism, $f:M\to M$, a point $x\in M$ is called \textit{nonwandering} if for any neighborhood $U$ of $x$, $\exists n\geq 1$ such that $f^n(U)\cap U\neq \emptyset$. The nonwandering set of $f$, denoted $\Omega(f)$ is the set of nonwandering points of $f$. The spectral decomposition admits the following generalization to Anosov homeomorphisms.
	
	\begin{theorem}[Spectral Decomposition, \cite{Aoki}]\label{thm:spectral decomp homeos}
		Let $f:M\to M$ be an Anosov homeomorphism of a compact manifold $M$. Then, there exist closed, pairwise disjoint sets $X_1,...,X_k$ and a permutation $\sigma\in S_k$ such that 
		\begin{enumerate}[(a)]
			\item $\Omega(f)=\bigcap_{i=1}^k X_i$,
			\item $f(X_i)=X_{\sigma(i)}$, and 
			\item if for $a>0$, $\sigma^a(i)=i$, then $f^a|_{X_i}$ is topologically mixing.
		\end{enumerate}
	\end{theorem}
	
	Recall that a continuous map $f:M\to M$ is \textit{topologically mixing} if for any open sets $U,V\subset M$, there exists an integer $N$ such that $f^n(U)\cap V\neq \emptyset$ for all $n\geq N$.

\subsection{Proof of Theorem \ref{thm:HiraideNilmanifolds}}

	The proof of Theorem \ref{thm:HiraideNilmanifolds} follows the same structure as the proofs of the main result of \cite{Hiraide_tori} with a couple of modifications to account for being on a nilmanifold instead of a torus. Before giving the details of the proof, we provide a brief synopsis of the proof and note where it differs from Hiraide's. The argument has three main parts.
	
		\begin{itemize}
		\item 
			Constructing the hyperbolic nilmanifold automorphism $A:M\to M$. This differs from Hiraide's argument \cite{Hiraide_tori} in the same ways that Manning's argument \cite{Manning_conjugacy} differs from Franks's \cite{Franks_Anosov_tori}. The construction of the nilmanifold automorphism $A:M\to M$ is the same as Manning's construction in \cite{Manning_conjugacy}. The proof that $A$ is hyperbolic follows Hiraide's argument using the same technique that Manning uses in \cite{Manning_nilmanifolds} and \cite{Manning_conjugacy} to get a formula for the Lefschitz number of $A$ in terms of the eigenvalues of $A$. 
			
		\item
			Constructing a semiconjugacy $h:M\to M$ between $A$ and $f$. Since $M$ is a $K(\pi,1)$, this construction is the same as that on the torus.
			
		\item
			Proving that the semiconjugacy $h:M\to M$ is actually a conjugacy. This follows the same argument given in Hiraide, with the main modification in Lemma \ref{lem:one fixed point} to construct a homotopy between $\tilde f$ and $\tilde A$ that doesn't introduce fixed points outside a compact set.
		\end{itemize}
		
	Now, we give the details of the proof. Let $f:M\to M$ be an Anosov homeomorphism of the nilmanifold $M=N/\Gamma$. We begin by finding a candidate for the hyperbolic nilmanifold automorphism in Theorem \ref{thm:HiraideNilmanifolds}. We'll do this following the same procedure as Franks \cite{Franks_Anosov_tori}, Manning \cite{Manning_conjugacy}, and Hiraide \cite{Hiraide_tori}; we'll find a `linear' model of $f$, which we'll then show is hyperbolic.  Our linear model of $f$ will be a nilmanifold automorphism $A$ that is homotopic to $f$. The construction of this linear model is identical to that given in \cite{Manning_conjugacy}. To construct $A$, we'll show that the induced action of $f$ on $\pi_1(M,e\Gamma)$ can be lifted to an automorphism of $\Gamma$. We'll then extend this automorphism to all of $N$ to get our linear model. 
		
	Let $f_*:\pi_1(M,e\Gamma)\to \pi_1(M,f(e\Gamma))$ be the homomorphism that $f$ induces on the fundamental group of $M$. We can view $\pi_1(M,e\Gamma)$ and $\pi_1(M,f(e\Gamma))$ as subgroups of $N$. To do this, we first identify $\pi_1(M,e)$ with $\G$ (via the endpoints of the lifts of the loops in the fundamental group). Recall that changing basepoint in the fundamental group is the same as conjugating by some path in $M$. So in the universal cover of $M$ (i.e. $N$), the identification that takes $\pi_1(M,e\Gamma)$ to $\Gamma$ will take $\pi_1(M,f(e\Gamma))$ to $x^{-1}\Gamma x$ for some $x\in N$. By lifting to $N$, we can view $f_*$ as a homomorphism $\Gamma\to x^{-1}\Gamma x$. 
		
	Since we want a homomorphism $\Gamma\to \Gamma$, we compose $f_*$ with conjugation by $x^{-1}$, which gives us our automorphism of $\Gamma$. To summarize, we've shown that we can lift $f_*:\pi_1(M,e\Gamma)\to \pi_1(M,f(e\Gamma))$ to an automorphism of $\Gamma$, which is defined up to an inner automorphism of $N$. We can uniquely extend $f_*:\Gamma\to \Gamma$ to an automorphism $\tilde A:N\to N$ \cite[Corollary 1 of Theorem 2.11]{Raghunathan}. Since $\tilde A$ preserves $\Gamma$, it descends to a nilmanifold automorphism, $A:M\to M$. Note that $f$ is homotopic to $A$ since they induce conjugate maps on $\pi_1(M)$ and $M$ is a $K(\pi,1)$. 
		
	We now claim that the linearization $A$ is hyperbolic, which follows immediately from the following proposition.
	
	\begin{prop}\label{prop: hyperbolicitylinearization}
		Let $f:M\to M$ be an Anosov homeomorphism of a nilmanifold $M=N/\G$. If $A:M\to M$ is a nilmanifold automorphism that is homotopic to $f$, then $A$ is hyperbolic. 
	\end{prop}
	
	\begin{proof}
		This proof is a combination of the techniques of Manning \cite[Theorem A]{Manning_conjugacy} and Hiraide \cite[Proposition 6.2]{Hiraide_tori}. By passing to a double cover of $M$, it suffices to consider the case where the unstable generalized foliation of $f$, $\F_f^u$, is orientable. The goal of this proof is to show that $A$ is hyperbolic. More formally, we need to show that $D_eA$ has no eigenvalues of absolute value one. Let $\lambda_1,...,\lambda_n$ be the eigenvalues of $D_eA$ (counted with multiplicity).
		
		The first step in this proof is to relate the number of $m$-periodic points of $f$, for $m\in \N$, to the eigenvalues of $D_eA$. We do this using the Lefschetz fixed point theorem. First, recall that since $f^m$ and $A^m$ are homotopic, their Lefschetz numbers are the same, i.e. $L(f^m)=L(A^m)$. Since $\F_f^u$ is orientable, the Lefschitz fixed point theorem and Proposition \ref{propB Hiraide} imply that the number of fixed points of $f^m$, denoted $N(f^m)$, is given by $N(f^m)=|L(f^m)|$ for sufficiently large $m$. Now, recall from \cite{Manning_nilmanifolds} that we can also write $L(f^m)=L(A^m)=\prod_{i=1}^n (1-\lambda_i^m)$. We've therefore shown that, for sufficiently large $m$, the number of fixed points of $f^m$ is given by
			\begin{equation}\label{eqn: number of fixed points eigenvalues}
				P_m(f)= N(f^m)=\prod_{i=1}^n \left|1-\lambda_i^m\right|.
			\end{equation}
		This equation cannot hold if $A$ is not hyperbolic by arguments given in the proof of \cite[Proposition 6.2]{Hiraide_tori}.
	\end{proof}
	
	Recall that when we defined the `linearization' $A$ of an Anosov homeomorphism $f:M\to M$ of a nilmanifold, we only were able to define $A$ up to an inner automorphism of $N$ because we didn't know whether $f$ had any fixed points. We are now equipped to show that $f$ does indeed have fixed points.
	
	\begin{corollary}
		An Anosov homeomorphism of a nilmanifold has at least one fixed point.
	\end{corollary}
	\begin{proof}
		This follows immediately from the Lefschetz fixed point theorem and Proposition \ref{prop: hyperbolicitylinearization}.
	\end{proof}
	
	Note that by conjugating the Anosov homeomorphism $f:M\to M$ by a translation, we can assume without loss of generality that $f$ fixes the point $e\Gamma\in M$. We let $A$ be the hyperbolic `linearization' of $f$ described above. 
	
	The goal of the rest of the proof is to construct a conjugacy between $A$ and $f$. To do this we first construct a semiconjugacy, $h:M\to M$, between $A$ and $f$.
	
	\begin{prop}\label{prop: semiconjugacy identity fixed pt}
		Let $M=N/\Gamma$ be a nilmanifold, and let $f:M\to M$ be a homeomorphism that fixes the point $e\Gamma\in M$. If $f$ is freely homotopic to a hyperbolic nilmanifold automorphism $A:M\to M$, then there exists a continuous map $h:M\to M$ (freely) homotopic to the identity such that $A\circ h=h\circ f$ and $h(e\Gamma)=e\Gamma$. Furthermore, the map $h$ is the unique map freely homotopic to the identity fixing $e\Gamma$.
	\end{prop}
	\begin{proof}
		Since $f$ and $A$ are freely homotopic, $M$ is a $K(\pi,1)$, and $A$ is hyperbolic, there exists a homomorphism $(h_0)_*:\pi_1(M,e\Gamma)\to \pi_1(M,e\Gamma)$, that is induced by a base point preserving map $h_0:M\to M$ that is freely homotopic to the identity, such that $A_*\circ (h_0)_*=(h_0)_*\circ f_*$. Under these conditions, \cite[Theorem 2.2]{Franks_Anosov_diffeos} states that there exists a unique continuous base point preserving map, $h:M\to M$, that is homotopic to $h_0$, such that $A\circ h=h\circ f$. 
	\end{proof}

	We complete the proof of Theorem \ref{thm:HiraideNilmanifolds} by proving that the semiconjugacy, $h:M\to M$, from Proposition \ref{prop: semiconjugacy identity fixed pt} is actually a conjugacy. To do this, we just need to show that $h$ is a homeomorphism. 
	
	\begin{prop}\label{prop:h_homeo}
		$h:M\to M$ is a homeomorphism.
	\end{prop}
	
	\begin{proof}
		The main step in this argument is to show that $h$ is a local homeomorphism. This combined with the fact that $h$ is surjective (because $h$ is homotopic to the identity and is a proper map) will imply that $h:(M,e\Gamma)\to (M,e\Gamma)$ is a covering map. Then, since $h$ is homotopic to the identity, the covering spaces $h:(M,e\Gamma)\to (M,e\Gamma)$ and $\id:(M,e\Gamma)\to (M,e\Gamma)$ are isomorphic, i.e. that there is a homeomorphism $g:M\to M$ such that $h=\id\circ g$. This will complete the argument that $h$ is a homeomorphism, and thus gives a conjugacy between $A$ and $f$. 
		
		Thus, all that remains is to show that $h$ is a local homeomorphism. We do this by showing that its lift $\tilde h:(N,e)\to (N,e)$
			\footnote{When we take this lift, we lift the point $e\Gamma\in M$ to the point $e\in N$. In the rest of this section, we'll be lifting $e\Gamma\in M$ to $e\in N$ unless otherwise noted.} 
		is a local homeomorphism. Recall that Brower's theorem on invariance of domain states that a locally injective continuous map between two manifolds without boundary of the same dimension is a local homeomorphism. Thus, we'll be done if we can show that $\tilde h$ is locally injective. In fact, we'll show that $\tilde h$ is injective. 
		
		First, we note that $f$ lifts to an Anosov homeomorphism $\tilde f:(N,e)\to (N,e)$. We recall from Section \ref{subsect:lifts} that the stable and unstable sets for $\tilde f$, denoted $\F_{\tilde f}^s$ and $\F_{\tilde f}^u$, are transverse generalized foliations on $N$. The first step in the argument that $\tilde h$ is injective is to show that it suffices to prove injectivity of $\tilde h$ on stable and unstable leaves of $\tilde f$. This follows from the fact that the stable and unstable generalized foliations for $\tilde f$ establish a global product structure for $N$, i.e. 
		
		\begin{prop}\label{prop:gproduct structure}
			For any points $x,y\in N$ the stable leaf through $x$ and the unstable leaf through $y$ intersect at exactly one point, i.e. the set $\tilde W^s(x)\cap \tilde W^u(y)$ contains exactly one point.
		\end{prop}
	
		Before going through the proof of Proposition \ref{prop:gproduct structure}, we show how this proposition implies that injectivity of $\tilde h$ follows from injectivity on stable and unstable leaves. This argument follows that in \cite[p.387-388]{Hiraide_tori}. Take $x,y\in N$ such that $\tilde h(x)=\tilde h(y)$. By Proposition \ref{prop:gproduct structure}, we can define a point $z:=\tilde W^s(x)\cap \tilde W^u(y)$ to be the intersection of the stable leaf through $x$ and the unstable leaf through $y$. If we show that $\tilde h(x)=\tilde h(y)=\tilde h(z)$, then injectivity of $\tilde h$ will follow from injectivity of the stable and unstable leaves. Thus, it suffices to show that $\tilde h(y)=\tilde h(z)$. 
		
		To prove $\tilde h(y)=\tilde h(z)$, recall that since $\tilde A$ is a hyperbolic automorphism of $N$, for arbitrary $M_1>0$, the map $\tilde A$ is expansive with expansive constant $M_1$. Thus, to show that $\tilde h(y)=\tilde h(z)$, it suffices to show that there exists a constant $M_1>0$ such that for all $n\in \Z$, 
			\begin{equation}\label{eqn:gproduct sufficient0}
				d\left( \tilde A^n \circ \tilde h(z), \tilde A^n\circ \tilde h(y) \right)\leq M_1.
			\end{equation}
		To see this, first recall that since $\tilde A\circ \tilde h=\tilde h\circ \tilde f$ and $\tilde h(x)=\tilde h(y)$, we have that for all $n\in \Z$,
			\begin{equation*}
					d\left( \tilde A^n \circ \tilde h(z), \tilde A^n\circ \tilde h(y) \right)
				=
					d\left( \tilde h\circ \tilde f^n(z),\tilde h\circ \tilde f^n(y) \right)
				=
					d\left( \tilde h\circ \tilde f^n(z),\tilde h\circ \tilde f^n(x) \right).
			\end{equation*}
		In light of these two equations, to prove \ref{eqn:gproduct sufficient0}, it's sufficient to prove that there exists a constant $M_1>0$ such that for all $n\geq 0$, the following two inequalities hold.
			\begin{align*}
					d\left( \tilde h\circ \tilde f^{-n}(z),\tilde h\circ \tilde f^{-n}(y) \right) \leq M_1\\
					d\left( \tilde h\circ \tilde f^n(z),\tilde h\circ \tilde f^n(x) \right)\leq M_1
			\end{align*}
		These inequalities follow immediately from the following two observations,
		\begin{itemize}
		\item 
			Since $h$ is homotopic to the identity, the map $\tilde h$ is a bounded distance away from the identity, i.e. there exists a constant $M_0>0$ such that $\forall w\in N$, $d\left( \tilde h(w),w \right)\leq M_0$.
		\item
			The facts that $z\in \tilde W^s(x)$ and $z\in \tilde W^u(y)$ imply that there exists a constant $C>0$ such that for all $n\geq 0$, 
				$$
					d\left( \tilde f^{n}(x), \tilde f^n(z) \right)\leq C
					\quad \text{ and } \quad 
					d\left( \tilde f^{-n}(y), \tilde f^{-n}(z) \right)\leq C
				$$
		\end{itemize}
	
	\end{proof}
	
	Now, all that remains is to prove that the stable and unstable generalized foliations give a global product structure on $N$, i.e. Proposition \ref{prop:gproduct structure}. The proof of this follows the proof of \cite[Lemma 6.8]{Hiraide_tori}, with a single minor change to account for the fact that $M$ is a nilmanifold instead of a torus. We therefore give the general steps in Hiraide's argument and note where modifications need to be made. The argument proceeds in four steps/lemmas. 
	
	\begin{lemma}\label{lem:one fixed point}
		Let $f:M\to M$ be an Anosov homeomorphism of the nilmanifold $M=N/\Gamma$. Let $\tilde f:N\to N$ be a lift of $f$ to $N$, and let $\tilde A:N\to N$ be a hyperbolic automorphism of $N$. If the $C^0$-distance between $\tilde A$ and $\tilde f$ is bounded, then $\tilde f$ has exactly one fixed point.
	\end{lemma}
	\begin{proof}
		This proof is a slight modification of the proof of \cite[Lemma 6.5]{Hiraide_tori}. The main ingredients in this proof are the Lefschetz fixed point theorem and the homotopy invariance of the Lefschetz number. Since we're working in a space that isn't compact, we need to be careful when using Lefschetz numbers.\footnote{Recall that the Lefschetz number of a map $g:X\to X$ is only defined if the set of fixed points $\Fix(g)$ is compact. Two maps have the same Lefschetz number if they are homotopic via a map that does not introduce fixed points out of a compact set. \cite{Dold}}
		Since $\tilde A$ is a hyperbolic automorphism, its Lefschetz number is $L(\tilde A)=\pm1$. The Lefschetz number of $\tilde f$ is defined because $\tilde f$ is a bounded distance away from $\tilde A$. 
		
		Now, we argue that $\tilde f$ has at least one fixed point. Since $N$ is contractible, we can construct a homotopy between $\tilde f$ and $\tilde A$ that does not introduce fixed points outside of a compact set. Thus, $L(\tilde f)=L(\tilde A)=\pm 1$, which implies that $\tilde f$ has at least one fixed point. 
		
		The fact that  $\tilde f$ has at most one fixed point follows from arguments in \cite[Lemma 6.7]{Hiraide_tori}.
	\end{proof}
	
	We now prove that the non-wandering set of $f$ is the whole nilmanifold. 
	
	\begin{lemma}\label{lem:nonwandering set whole manifold}
		The nonwandering set of an Anosov homeomorphism $f:M\to M$ of a nilmanifold $M=N/\Gamma$ is the entire nilmanifold, i.e. $\Omega(f)=M$. 
	\end{lemma}
	\begin{proof}
		This follows from the same argument as \cite[Proposition 6.6]{Hiraide_tori}. 
	\end{proof}
	
	We begin to show that the stable and unstable generalized foliations of $\tilde f$ give a global product structure on $N$.  
	
	\begin{lemma}\label{lem:intersection at most 1 pt}
		For $x,y\in N$, the stable manifold of $\tilde f$ at $x$, $\tilde W^s(x)$, and the unstable manifold of $\tilde f$ at $y$, $\tilde W^u(y)$, intersect at at most one point.
	\end{lemma}
	\begin{proof}
		This follows from the previous two lemmas along with the spectral decomposition. The details are exactly the same as those in \cite[Lemma 6.7]{Hiraide_tori}.
	\end{proof}
	
	Now to complete the proof of Proposition \ref{prop:gproduct structure} we just need to show that $\tilde W^s(x)$ and $\tilde W^u(y)$ actually intersect for each $x,y\in N$. This follows by gluing together local product neighborhoods given by $\F_f^s$ and $\F_f^u$ using the arguments in \cite[Lemma 1.6]{Franks_Anosov_tori}. 

	Now, all that remains in the proof of Theorem \ref{thm:HiraideNilmanifolds} is to show that $\tilde h$ is injective on the stable and unstable leaves of $\tilde f$, which proceeds exactly as in \cite{Hiraide_tori}.

	\bibliographystyle{apalike}
	\bibliography{SmoothModelsPaperReferences}

\end{document}